\documentclass[12pt]{article}
\usepackage{amsmath,amsthm,amssymb}
\usepackage{bbm}
\usepackage{color}
\usepackage{epsfig,epstopdf}
\usepackage{enumerate}

\usepackage[a4paper,margin=2cm]{geometry}

\newtheorem{thm}{Theorem}
\newtheorem{lem}[thm]{Lemma}
\newtheorem{prop}[thm]{Proposition}

\newcommand{\R}{\mathbbm{R}}
\newcommand{\N}{\mathbbm{N}}
\newcommand{\DEF}{\sl}
\newcommand{\conv}{\mathop{\mathrm{conv}}}
\newcommand{\PT}{\mathop{\mathrm{P}_\mathrm{traveling\ salesman}}}
\newcommand{\PM}{\mathop{\mathrm{P}_\mathrm{perfect\ matching}}}

\newcommand{\xc}{\mathop{\mathrm{xc}}} 

\newcommand{\setDef}[2]{\{{#1} \mid {#2}\}}

\DeclareMathOperator{\suppOp}{supp}
\newcommand{\supp}[1]{\suppOp(#1)}

\newcommand{\sabs}[1]{{\lvert{#1}\rvert}}
\title{Uncapacitated Flow-based Extended Formulations}
\author{Samuel Fiorini\footnote{Partially supported by \emph{Fonds National de la Recherche Scientifique} (F.R.S.-FNRS) and the \emph{Actions de Recherche Concert\'ees} (ARC) fund of the French community of Belgium.} \and Kanstantsin Pashkovich\footnote{Supported by the Progetto di Eccellenza 2008-2009 of the Fondazione Cassa Risparmio di Padova e Rovigo.}}

\begin{document}

\maketitle

\begin{abstract}
An extended formulation of a polytope is a linear description of this polytope using extra variables besides the variables in which the polytope is defined. The interest of extended formulations is due to the fact that many interesting polytopes have extended formulations with a lot fewer inequalities than any linear description in the original space. This motivates the development of methods for, on the one hand, constructing extended formulations and, on the other hand, proving lower bounds on the sizes of extended formulations. 

Network flows are a central paradigm in discrete optimization, and are widely used to design extended formulations. We prove exponential lower bounds on the sizes of uncapacitated flow-based extended formulations of several polytopes, such as the (bipartite and non-bipartite) perfect matching polytope and TSP polytope. We also give new examples of flow-based extended formulations, e.g., for 0/1-polytopes defined from regular languages. Finally, we state a few open problems.
\end{abstract}

\section{Introduction}


An {\em extended formulation} (shorthand: {\em EF}) of a polytope $P \subseteq \mathbb{R}^d$ is a system of linear constraints
\begin{equation}
\label{eq:EF}
E^\leqslant x + F^\leqslant y \leqslant g^\leqslant, \quad
E^= x + F^= y = g^=
\end{equation}
with $(x,y) \in \mathbb{R}^{d+k}$ such that $x \in \mathbb{R}^d$ belongs to $P$ if and only if there exists $y \in \mathbb{R}^k$ such that $(x,y)$ satisfies \eqref{eq:EF}. An extended formulation of $P$ is simply a linear description of $P$ in an extended space. Geometrically, $P$ is described as the projection of the polyhedron\footnote{We remark that although we allow for now $Q$ to be unbounded, we will soon show that one can restrict to the case where $Q$ is bounded, that is, a polytope.} $Q \subseteq \mathbb{R}^{d+k}$ defined by \eqref{eq:EF}. More generally, we call a polyhedron $Q \subseteq \mathbb{R}^e$ an {\em extension} (or {\em lift}) of $P$ if there exists an affine map $\pi : \mathbb{R}^e \to \mathbb{R}^d$ such that $\pi(Q) = P$. 

Consider a linear description $Ax \leqslant b$ of $P$ in its original space. If $f : \mathbb{R}^d \to \mathbb{R}$ is any function, then 
\begin{equation}
\label{eq:opt_problem_EF}
\sup \setDef{f(x)}{Ax \leqslant b} = \sup \setDef{f(x)}{E^\leqslant x + F^\leqslant y \leqslant g^\leqslant, \ E^= x + F^= y = g^=}\,.
\end{equation}
Thus every optimization problem on $P$ can be reformulated as an optimization problem over any extension of $P$. This is why extended formulations are interesting for optimization: in  \eqref{eq:opt_problem_EF}, the number of constraints in the right-hand side can be much smaller than the number of constraints in the left-hand side. 

We define the {\em size} of an extended formulations as its number of inequalities, and the size of an extension as its number of facets; these turn out to be the right measures of size. Note that the size of an extended formulation is at least the size of the associated extension because every facet of a polyhedron is part of every linear description of this polyhedron (in the space in which it is defined), and every extension corresponds to an extended formulation with exactly its size.

The field of extended formulations is attracting more and more attention. In particular, size lower-bounding techniques are becoming increasingly powerful and diverse, see, e.g., \cite{Yannakakis91,
KaibelPashkovichTheis10,GP12,FioriniKaibelPashkovichTheis11,extform4,BFPS12,BM13,BP2013commInfo}. The reader will find in the surveys~\cite{CCZ10,Kaibel11,Wolsey11} a good description of the field as it was a few years ago.

In this paper, we study some restricted forms of extended formulations (extensions) which we call {\em flow-based extended formulations (extensions)}, see Section \ref{sec:flow-based_EFs} for a definition. Informally, a flow-based extension of a polytope $P$ is another polytope $Q$ that can be realized as the convex hull of all flows in some network. This definition is inspired by the prominent role played by network flows in discrete optimization: many algorithms and structural results crucially rely on network flows~\cite{AhujaMagnantiOrlinBook,SchrijverBook}. Quite a lot of known extended formulations are based on network flows, such as those obtained from dynamic programming algorithms~\cite{Martin91}.

Here, we focus on uncapacitated networks. Our main contribution is to prove size lower bounds of the form $2^{\Omega(n)}$ for uncapacitated  flow-based extended formulations of several polytopes, such as the perfect matching polytope of (bipartite and non-bipartite) complete graphs and the traveling salesman polytope of the complete graph. Our results are summarized in Table~\ref{tab:results}. Below, the notations $O^*(\cdot)$, $\Omega^*(\cdot)$ and $\Theta^*(\cdot)$ have the same meaning as the usual notations $O(\cdot)$, $\Omega(\cdot)$ and $\Theta(\cdot)$, except that polynomial factors are ignored.

\begin{table}[ht]
\centering
\begin{tabular}{|r||c|c|}
\hline
Polytope & Size bounds for general EFs & Size bounds for flow-based EFs\\
\hline
\hline
$\PM(K_{n,n})$ & $\Theta(n^2)$~\cite{Birkhoff} & $\mathbf{\Theta^*(2^n)}$\\
\hline
$\PM(K_n)$ & $\Omega(n^2)$, $O^*(2^{\frac{n}{2}})$~\cite{KaibelPashkovichTheis10,FaenzaFioriniGrappeTiwary11}
 & $\mathbf{\Omega^*(2^{\frac{n}{2}})}$, $\mathbf{O(2^{0.695n})}$\\
\hline
$\PT(K_n)$ & $2^{\Omega(\sqrt{n})}$~\cite{extform4}, $O^*(2^n)$~\cite{heldKarp70} &$\mathbf{\Omega^*(2^{\frac{n}{4}})}$, $O^*(2^n)$~\cite{heldKarp70}\\
\hline
\end{tabular}
\caption{Table of results. New results are indicated in boldface. The bounds for flow-based EFs assume that the network is uncapacitated.} \label{tab:results}
\end{table}

Before giving an outline of the paper, we briefly discuss our motivations. Lower bounds on restricted types of extended formulations have been studied by quite many authors, starting with the work of Yannakakis~\cite{Yannakakis91} on symmetric extended formulations. There has been work on hierarchies such as the Sherali-Adams~\cite{SheraliAdams1990} and Lov\'asz-Schrijver hierarchies~\cite{LovaszSchriver1991}, see, e.g., \cite{BGHMT2006,STT2007,FernandezdelaVegaMathieu2007,CMM2009,GMT2009,BenabbasMagen2010}; further work on symmetric extended formulations~\cite{KaibelPashkovichTheis10,Pashkovich12,GP12} and also work on extended formulations from low variance protocols~\cite{FaenzaFioriniGrappeTiwary11}. 

We think that the restriction of being flow-based is as natural as the restrictions studied in the aforementioned papers. Combinatorial optimization offers a variety of modeling tools beyond flows, which are the most basic and important modeling tool: e.g., matchings, polymatroids and polymatroid intersections~\cite{SchrijverBook}. It seems a worthy research goal to characterize the expressivity of these modeling tools, and give theoretical explanations of the fact that some problems can be efficiently expressed by some modeling tools and not by others. This paper is a first step in that direction. 

Of particular interest are {\em separations} between modeling tools. It is striking that all our lower bounds rely on a separation between uncapacitated and capacitated flows: while the perfect matching polytope of the complete bipartite graph $K_{n,n}$ has a $O(n^2)$-size capacitated flow-based extended formulation, we show a $\Omega^*(2^n)$ lower bound on the size of every uncapacitated flow-based extended formulations of that polytope. Via reductions, we derive from this the other lower bounds reported in Table~\ref{tab:results}.

We conclude this discussion by focussing on the traveling salesman polytope. Held and Karp~\cite{heldKarp70} gave a $O^*(2^n)$-complexity dynamic programming algorithm for the traveling salesman problem based on subsets. In our terminology, this yields a $O^*(2^n)$-size uncapacitated flow-based extended formulation for the traveling salesman polytope. In a survey paper on exact algorithms for combinatorial optimization problems, Woeginger~\cite{Woeginger03} stated as an open problem the question of determining if the traveling salesman problem has an exact algorithm of complexity $(2-\varepsilon)^n$ for some $\varepsilon > 0$. The question was answered affirmatively by Bjorklund~\cite{Bjorklund10}, at least if one tolerates randomized algorithms with small failure probability and restricts to instances where the coefficients are bounded. Our $\Omega^*(2^{\frac{n}{4}})$ lower bound for uncapacitated flow-based extended formulations for the traveling salesman polytope also applies to dynamic programming algorithms for the traveling salesman problem, which sheds some light on Woeginger's question.

The rest of the paper is organized as follows. We begin with preliminaries in Section~\ref{sec:preliminaries}: after introducing some notations, we define convex polytopes in general as well as the particular convex polytopes studied here. Then, in Section~\ref{sec:flow-based_EFs}, we formally define flow-based extended formulations, discuss an example and establish basic properties of flow-based extended formulations, focussing on the uncapacitated case. Finally, in Section \ref{sec:lower_bounds}, we prove size bounds for uncapacitated flow-based extended formulations described in Table \ref{tab:results}.

\section{Preliminaries} \label{sec:preliminaries}

Let $I$ be a finite ground set. The {\DEF incidence vector} of a subset $J \subseteq I$ is the vector $\chi^J \in \mathbb{R}^I$ defined as 
\[
  \chi^J_i = \left\{
  \begin{array}{l l}
    1 & \quad \text{if } i \in J\\
    0 & \quad \text{if } i \notin J
  \end{array} \right.
\]
for $i \in I$. For $x \in \mathbb{R}^I$, we let $x(J) := \sum_{i \in J} x_i$.

First, let $G = (V,E)$ be an undirected graph. For a subset of vertices $U\subseteq V$, we denote as $\delta(U)$ the set of edges of $G$ with exactly one endpoint in $U$. So,
\begin{eqnarray*}
\delta(U) &= &\{uv \in E : u \in U, v \notin U\}\ .
\end{eqnarray*}

Now, let $N=(V,A)$ be a directed graph. For $U \subseteq V$, we denote by $\delta^+(U)$ the set of arcs of $N$ with tail in $U$ and head in $V \setminus U$, and by $\delta^{-}(U)$ the set of arcs of $N$ with head in $U$ and tail in $V\setminus U$, i.e.
\begin{eqnarray*}
\delta^+(U) &= &\{(u,v) \in A : u\in U, v\notin U\}\ , \text{ and}\\
\delta^-(U) &= &\{(v,u) \in A : u\in U, v\notin U\}\ .
\end{eqnarray*}
As usual, for $v \in V$, we use the shortcuts $\delta(v)$, $\delta^+(v)$ and $\delta^-(v)$ for $\delta(\{v\})$, $\delta^+(\{v\})$ and $\delta^-(\{v\})$ respectively.

\subsection{Convex Polytopes and Polyhedra} \label{sec:convex_polytopes}

A {\em (convex) polytope} is a set $P \subseteq \mathbb{R}^d$ that is the convex hull of a finite set of points  in $\mathbb{R}^d$. Equivalently, $P \subseteq \mathbb{R}^d$ is a polytope if and only if $P$ is bounded and the intersection of a finite collection of closed halfspaces. This is equivalent to saying that $P$ is bounded and the set of solutions of a finite system of linear inequalities (or equalities, each of which can be represented by a pair of inequalities). A {\em (convex) polyhedron} is similar to a polytope, except that it may be unbounded. Formally, a polyhedron $Q \subseteq \mathbb{R}^d$ is any set that can be represented as the Minkowski sum of a polytope and a polyhedral cone or, equivalently, as the intersection of a finite collection of closed halfspaces. For more background on polytopes and polyhedra, see the standard reference~\cite{Ziegler}.

\subsection{Perfect Matching Polytope} \label{sec:perfect_matching_polytope}

A {\DEF perfect matching} of an undirected graph $G=(V,E)$ is set of edges $M \subseteq E$ such that every vertex of $G$ is incident to exactly one edge in $M$. The {\DEF perfect matching polytope} of the graph $G$ is the convex hull of the incidence vectors of the perfect matchings of $G,$ i.e., 
$$
\PM(G) = \conv\{\chi^M \in\mathbb{R}^E : M~ \text{perfect matching of}~ G\}\ .
$$
Edmonds \cite{Edmonds65} showed that the perfect matching polytope of $G$ is described by the following system of linear constraints (see also \cite{SchrijverBookA03}, page 438):
\begin{eqnarray}
\label{eq:odd_cut} x(\delta(U)) &\geqslant &1 \quad \text{for } U\subseteq V \text{ with } |U| \text{ odd}\ ,\\
\nonumber x(\delta(v)) &= &1 \quad \text{for } v \in V\ ,\\
\nonumber x_e &\geqslant& 0 \quad \text{for } e\in E\ .
\end{eqnarray}

In the case where the graph $G$ is bipartite, that is, when the vertex set $V$ can be partitioned into two sets $A$ and $B$ such that every edge in $E$ has an endpoint in $A$ and the other in $B$, the odd cut inequalities \eqref{eq:odd_cut} may be dropped~\cite{Birkhoff}. Thus the perfect matching polytope of a bipartite graph $G$ is described as follows:
\begin{eqnarray*}
 x(\delta(v)) &= &1 \quad \text{for } v \in V\ ,\\
 x_e &\geqslant& 0 \quad \text{for } e\in E\ .
\end{eqnarray*}

\subsection{Traveling Salesman Polytope} \label{sec:traveling_salesman_polytope}

A {\DEF Hamiltonian cycle} of $G=(V,E)$ is a connected subgraph of $G$ such that every vertex of $G$ is incident to exactly two edges in $C$. The {\DEF traveling salesman polytope} of the graph $G$ is the convex hull of the incidence vectors of the hamiltonian cycles of $G,$ i.e., 
$$
\PT(G) = \conv\{\chi^{E(C)} \in\mathbb{R}^E : C~ \text{Hamiltonian cycle of}~ G\}\ .
$$
In the formula above, $E(C)$ denotes the edge set of Hamiltonian cycle $C$.

No linear description of the traveling salesman polytope of the complete graph $K_n$ is known. Moreover no ``reasonable'' linear description of this polytope should be expected unless $\mathcal{NP}=\text{co-}\mathcal{NP}$ (see Corollary 5.16a \cite{SchrijverBookA03}).

\subsection{Flow Polyhedron} \label{sec:flow_polyhedra}

Let $N = (V,A)$ be a network with source node $s \in V$, sink node $t \in V \setminus \{s\}$ and arc capacities $c_a \in \mathbb{R}_+ \cup \{\infty\}$ for $a \in A$. An $s$--$t$ {\em flow} of value $k$ is a vector $\phi \in \mathbb{R}^A$ satisfying 
\begin{eqnarray}
\label{eq:flow_balance}
\phi(\delta^+(v)) - \phi(\delta^-(v)) &= &0 \quad \forall v \in V \setminus \{s,t\},\\
\label{eq:flow_value}
\phi(\delta^+(s)) - \phi(\delta^-(s)) &= &k,\\
\label{eq:flow_lb}
\phi_a &\geqslant &0 \quad \forall a \in A,\\
\label{eq:flow_ub}
\phi_a &\leqslant &c_a \quad \forall a \in A.
\end{eqnarray}
For a fixed $k \in \mathbb{R}$, the set of all $s$--$t$ flows of value $k$ in network $N$ defines a polyhedron $Q = Q(V,A,s,t,k,c)$ that we call {\em flow polyhedron}. 

In this paper, we will assume most of the time that the network is {\em uncapacitated}, that is, $c_a = \infty$ for all $a \in A$. This amounts to ignoring the upper bound inequalities \eqref{eq:flow_ub}.

\section{Flow-based Extended Formulations} \label{sec:flow-based_EFs}

\subsection{Definition} \label{sec:flow-based_EFs_def}

Consider again a network $N = (V,A)$ with source node $s \in V$, sink node $t \in V \setminus \{s\}$, arc capacities $c_a \in \mathbb{R}_+ \cup \{\infty\}$ for $a \in A$ and flow value $k \in \mathbb{R}_+$. We say that the flow polyhedron $Q = Q(V,A,s,t,k,c)$ is a {\em flow-based extension} of a given polytope $P$ in $\mathbb{R}^d$ if there exists a linear projection $\pi : \mathbb{R}^A \to \mathbb{R}^d$ such that $\pi(Q) = P$. 
A flow-based extension is said to be {\em uncapacitated} if the associated network is uncapacitated. 

From now on, we will always assume that the projection $\pi$ is linear. This causes essentially no loss of generality because an affine projection can be made linear at the cost of adding one new arc $(s',s)$ to the network and moving the source to the node $s'$. We denote by $M \in \mathbb{R}^{d \times A}$ the matrix of projection $\pi$, that is, the matrix $M\in \mathbb{R}^{d \times A}$ such that $\pi(\phi) = M\phi$ for all $\phi \in \mathbb{R}^A$. 

Moreover, we denote by $F \in \mathbb{R}^{(V \setminus \{s,t\}) \times A}$ the coefficient matrix of the flow balance equations. In other words, $F\phi = 0$ is the matrix form of \eqref{eq:flow_balance}. Then, the flow-based extension $Q$ can be described algebraically as:
\begin{equation}
\label{eq:flow_EF}
x = M\phi,\ F\phi = 0,\ \phi(\delta^+(s)) - \phi(\delta^-(s)) = k,\ 0 \leqslant \phi \leqslant c,
\end{equation}
We call system~\eqref{eq:flow_EF} a {\em flow-based extended formulation} of $P$.  

Notice that in the uncapacitated case, the size (that is, number of inequalities) of a flow-based extended formulation is exactly the number of arcs in the corresponding network.

Notice also that in the uncapacitated case, we can assume that $k = 1$ without loss of generality. This is because changing $k$ to $1$ simply amounts to replacing $Q$ by $(1/k)Q$. Indeed, if $\pi : \mathbb{R}^A \to \mathbb{R}^d$ projects $Q$ to $P$, then $\pi' :  \mathbb{R}^A \to \mathbb{R}^d : \phi \mapsto \pi'(\phi) := \pi(k \phi)$ projects $(1/k)Q$ to $P$. (In case $k = 0$, $Q$ is just a point. We will ignore this case in what follows.)

We will prove below that in the uncapacitated case, we can furthermore assume that $N$ is acyclic, provided $\varnothing \subsetneq P \subseteq \mathbb{R}^d_+$. In this case, $Q$ is a polytope and its vertices are the characteristic vectors $\chi^\sigma$ of all directed $s$--$t$ paths $\sigma$ in network $N$ (this follows from the well-known fact that the system \eqref{eq:flow_balance}--\eqref{eq:flow_lb} defining $Q$ is totally unimodular). We call such an extension an {\em $s$--$t$ path extension}, any corresponding extended formulation an {\em $s$--$t$ path extended formulation} and define the {\em $s$--$t$ path extension complexity} $\xc_\text{$s$--$t$ path}(P)$ of a polytope $P$ as the minimum number of arcs of a network whose $s$--$t$ path polytope is an extension of $P$. We will show that this is also the minimum size of an uncapacitated flow-based extended formulation of $P$.


%

\subsection{Example: Regular Languages}

In order to convince the reader that $s$--$t$ path extensions are quite powerful, we now discuss an illustrating example that generalizes Carr and Konjevod's flow-based extended formulation of the convex hull of even 0/1-vectors in $\mathbb{R}^n$~\cite{CK04}.

Consider a {\DEF deterministic finite automaton} $M$ over the alphabet $\{0,1\}$, that is, a $4$-tuple $(Q,\delta,q_0,F)$ where $Q$ is now a (nonempty) finite set of {\DEF states}, $\delta : Q \times \{0,1\} \to Q$ is the {\DEF transition function}, $q_0 \in Q$ is the {\DEF initial state} and $F \subseteq Q$ is the set of {\DEF accept states}. For a given input word $x = x_1 x_2 \cdots x_n$ in $\{0,1\}^*$, the automaton $M$ performs a computation starting at the initial state $q_0$ and in which the state $q_{i}$ ($i \in [n]$) is determined by the previous state $q_{i-1}$ and the $i$th letter $x_i$ of word $x$ through the equation $q_{i} = \delta(q_{i-1},x_i)$. The automaton is said to {\DEF accept} $x$ if the final state $q_n$ is an accept state, that is, $q_n$ belongs to $F$. 

The automaton $M$ defines a language $L = L(M)$ over $\{0,1\}$ consisting of all words $x \in \{0,1\}^*$ accepted by $M$. Such a language is said to be {\DEF regular}. Now pick a positive integer $n$, and consider a word $x = x_1x_2 \cdots x_n$ of length $n$ in $L$. Treating each letter of word $x$ as belonging to a different coordinate, we see that $x$ defines a $0/1$-vector $(x_1,x_2,\ldots,x_n)^\intercal$ in $\mathbb{R}^n$. By taking the convex hull of all $0/1$-vectors corresponding to all words of length $n$ in $L$, we obtain a $0/1$-polytope $P_n(L)$ in $\mathbb{R}^n$.

As we show now, one can easily construct compact flow-based extended formulations for such $0/1$-polytopes.

\begin{prop}
Let $L$ denote a regular language over $\{0,1\}$ and $M = (Q,\delta,q_0,F)$ any deterministic finite automaton recognizing the language $L$. For each positive integer $n$, there exists an $s$--$t$ path extended formulation of $P_n(L)$ with size at most $2|Q|n$.
\end{prop}
\begin{proof}
We define a network $N$ from automaton $M$. Besides source node $s$ and sink node $t$, network $N$ has $n-1$ nodes $(q,1)$, \ldots, $(q,n-1)$ for each state $q \in Q$. To simplify notations, we also denote $s$ by $(q_0,0)$. This defines the node set $V$ of $N$. For $i \in [n-1]$, we connect node $(q,i-1)$ to each of the nodes $(\delta(q,0),i)$ and $(\delta(q,1),i)$ by an arc.  Moreover, for each transition $q ' = \delta(q,\sigma)$ with $q' \in F$ we add an arc from node $(q,n-1)$ to sink node $t$. This defines the arc set $A$ of $N$. See Figure \ref{fig:even} for an example. In a formula, we have (with a slight abuse of notation because the network can have parallel arcs)
\begin{eqnarray*}
V &= &\{\underbrace{(q_0,0)}_{= s}\} \cup \setDef{(q,i)}{q \in Q, i \in [n-1]} \cup \{t\},\\
A &= &\setDef{((q,i-1),(\delta(q,\sigma),i))}{(q,i-1) \in N, i \in [n-1], \sigma \in \{0,1\}}\\ 
&&\mbox{} \cup \setDef{((q,n-1),t)}{\exists \sigma \in \{0,1\} : \delta(q,\sigma) \in F}\,.
\end{eqnarray*}
Each arc $a \in A$ corresponds to a transition $q' = \delta(q,\sigma)$, and is said to carry the label $\sigma \in \{0,1\}$. Thus the label carried by an arc is the symbol that caused the transition.

\begin{figure}[ht]
\centering
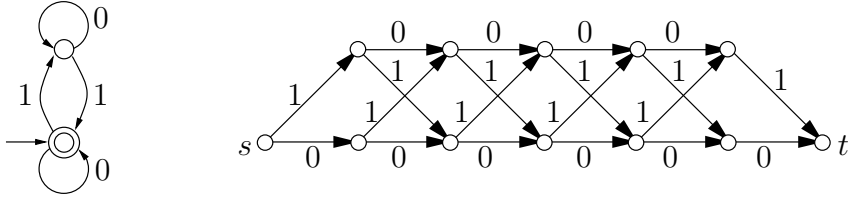
\caption{Deterministic finite automaton (left) and corresponding network (right).}
\label{fig:even}
\end{figure}

In the network $N=(V,A)$, we send $k = 1$ units of flow from $s$ to $t$, setting all capacities $c_a$ to $\infty$. The column of the projection matrix corresponding to arc $a \in A$ from node $(q,i-1)$ is the $0/1$-vector $(0,\ldots,0,\sigma,0,\ldots,0)^\intercal$ with $\sigma$ in position $i$ and $0$ everywhere else, where $\sigma \in \{0,1\}$ is the label carried by arc $a$. We leave it to the reader to perform the straightforward check that this defines an $s$-$t$ path extended formulation of $P_n(L)$.

The size of this extended formulation is the number of arcs in the network, that is,
$$
2 + 2|Q|(n-1) \leqslant 2 |Q| n.
$$
\end{proof}


\subsection{Basic Properties}

\subsubsection{Nonnegativity of the Projection}

A linear projection $\pi : \mathbb{R}^A \to \mathbb{R}^d$ is called {\em nonnegative} if its projection matrix is (entry-wise) nonnegative.

\begin{lem} \label{lem:nonnegative_pi}
For every uncapacitated flow-based extension $Q \subseteq \mathbb{R}^A$, $\pi : \mathbb{R}^A \to \mathbb{R}^d$ of a polytope $P \subseteq \mathbb{R}^d_+$, there is a nonnegative linear projection $\pi' : \mathbb{R}^A \to \mathbb{R}^d$ such that $\pi'(Q) = P$.
\end{lem}
\begin{proof}
As above, let $M$ denote the matrix of $\pi$. It suffices to show that for every row $M_i$ of the matrix $M$ there exists a row vector $\Lambda_i \in (\mathbb{R}^{V \setminus \{s,t\}})^*$ such that $M_i + \Lambda_i F \geqslant 0$, since due to~\eqref{eq:flow_EF} the system $F\phi = 0$ holds for all $\phi\in Q$ and thus $(M + \Lambda F) \phi = M \phi + \Lambda F \phi = M\phi$.

Suppose, for the sake of contradiction, that no such $\Lambda_i$ exists for some $i$. Then by Farkas' lemma, there exists a vector $\psi \in\mathbb{R}^{A}$ such that 
\begin{equation*}
 F\psi = 0,\quad \psi \geqslant 0 \quad \text{and} \quad M_i \psi < 0\,.
\end{equation*}
Thus $\psi$ is an $s$--$t$ flow in $N$. Because the network is uncapacitated, we can assume that the value of $\psi$ is precisely $k$, by scaling $\psi$ if necessary, hence $\psi \in Q$. Now, the inequality $M_i \psi < 0$ means that the $i$th coordinate of the projection $\pi(\psi) = M\psi$ is negative, which gives the desired contradiction.
\end{proof}

\subsubsection{Acyclicity of the Network}


\begin{lem} \label{lem:acyclic}
The network associated to every minimum size uncapacitated flow-based extension $Q \subseteq \mathbb{R}^A$ of a nonempty polytope $P \subseteq \mathbb{R}_+^d$ is acyclic.
\end{lem}
\begin{proof}
By Lemma \ref{lem:nonnegative_pi} the projection $\pi : \phi \mapsto M\phi$ may be assumed nonnegative. Consider a directed cycle $C$ in network $N$ and the corresponding columns of $M$. Take a point $\phi\in Q$ and consider the projection $\pi(\phi+K\chi^C)$ where $K \in \mathbb{R}_+$. By linearity, $\pi(\phi+K\chi^C) = \pi(\phi) + K \pi(\chi^C)$. If $\pi(\chi^C)$ is a non-zero vector and $K$ is chosen large enough, $\pi(\phi)+K \pi(\chi^C)$ would be outside of polytope $P$, a contradiction to the fact that $\phi+K\chi^C$ satisfies~\eqref{eq:flow_EF} and thus lies in $Q$. 

Hence $\pi(\chi^C)$ is a zero vector. Due to nonegativity of $\pi$, for every arc $a\in A$ contained in at least one directed cycle, the corresponding column of $M$ is zero, that is, $\pi(\chi^{\{a\}}) = 0$. Therefore, if $N$ contains a directed cycle, we can contract every strongly connected component of $N$ to a node and obtain a smaller flow-based extension of $P$, a contradiction. Note that if $s$ and $t$ are in the same strongly connected component of $N$, in which case we are not allowed to contract this component because we assume $s \neq t$, then necessarily $P = \{0\}$ and a minimum size flow-based extension of $P$ is given by a network with two nodes connected by a single arc. The result follows.
\end{proof}

\subsubsection{Equations for the Initial Polytope}

\begin{lem}\label{lem:equations}
Let the equation $c\,x= \delta$ be valid for a nonempty polytope $P \subseteq \mathbb{R}^d$. Then for every node $v$ in the network $N=(V,A)$ associated to a minimum-size uncapacitated flow-based extension $Q \subseteq \mathbb{R}^A$ of $P$, there is a unique $\epsilon\in\R$ such that $c \, \pi(\chi^\sigma)=\epsilon$ for every $s$--$v$ path $\sigma$.
\end{lem}
\begin{proof}
 Let $\sigma_1, \sigma_2$ be two paths from source $s$ to node $v$. Due to minimality of the extension there is also a path $\sigma_3$ from $v$ to $t$. Since $\sigma_1\cup \sigma_3$ and $\sigma_2\cup \sigma_3$ define paths from $s$ to $t$, the projections $\pi(\chi^{\sigma_1\cup \sigma_3})$ and $\pi(\chi^{\sigma_2\cup \sigma_3})$ lie in the polytope $P$, and thus satisfy the equation $c\, x = \delta$. Therefore,
 \begin{equation*}
   0= c \, \pi(\chi^{\sigma_1\cup \sigma_3})-c \, \pi(\chi^{\sigma_2\cup \sigma_3})=c \, \pi(\chi^{\sigma_1}) - c \, \pi(\chi^{\sigma_2})\,.
\end{equation*}
To conclude the proof, we may define $\epsilon$ as the value $c \, \pi(\chi^{\sigma_1})$.
\end{proof}

\subsubsection{Extension of Faces}

\begin{lem}\label{lem:faces}
  For every polytope $P \neq \varnothing$ and face $F$ of $P$, there holds $\xc_\text{$s$--$t$ path}(P)\geqslant \xc_\text{$s$--$t$ path}(F)$.
\end{lem}
\begin{proof}
   Let $Q$ be a minimum size $s$--$t$ path extension of $P$ and let $N=(V,A)$ denote the corresponding network. The polytope $\pi^{-1}(F)\cap Q$ is a face of $Q$. From the linear description of $Q$, see \eqref{eq:flow_balance}--\eqref{eq:flow_lb}, we infer 
  \begin{equation*}
      \pi^{-1}(F)\cap Q=\setDef{\phi\in Q}{\phi_a=0\,, a\in A'}
  \end{equation*}
for some $A'\subseteq A$. Hence, the $s$--$t$ path polytope $Q'$ associated with the network $N'=(V, A\setminus A')$ together with the projection $\pi$ defines an $s$--$t$ path extension of face $F$. Because the size of the extension $Q'$ of $F$ is not larger than the size of the extension $Q$ of $P$, we have 
$\xc_\text{$s$--$t$ path}(F) \leqslant \xc_\text{$s$--$t$ path}(P)$.
\end{proof}
 
\section{Lower Bounds} \label{sec:lower_bounds}

Now we provide lower bounds on the size of uncapacitated flow-based extensions or, equivalently (by Lemmas \ref{lem:nonnegative_pi} and \ref{lem:acyclic}), $s$--$t$ path extensions of the (bipartite and non-bipartite) perfect matching polytope and traveling salesman polytope. We start by proving that the $s$--$t$ path extension complexity of the perfect matching polytope of $K_{n,n}$ is $\Theta^*(2^n)$. This is striking because this polytope has $\Theta(n^2)$ facets, and a size-$\Theta(n^2)$ {\em capacitated} flow-based extension. Perhaps less striking are our exponential lower bounds for the perfect matching polytope and traveling salesman polytope of $K_n$. We derive these by combining our lower bound on $\xc_{\text{$s$--$t$ path}}(\PM(K_{n,n}))$ and Lemma \ref{lem:faces}.

\subsection{Bipartite Perfect Matchings}

\begin{thm}\label{thm:bipartite_matchings}
 Every uncapacitated flow-based extension (or, equivalently, $s$--$t$ path extension) of the perfect matching polytope of the complete bipartite graph $K_{n,n}$ has size $\Omega\left(\frac{2^{n}}{\sqrt{n}}\right)$.
\end{thm}
\begin{proof}
  Due to Lemma~\ref{lem:nonnegative_pi}, we may assume that the projection $\pi:\R^{A}\rightarrow\R^{d}$ is given by a linear nonnegative map.
  Consider an $s$--$t$ path extension $Q\subseteq\R^{A}$ with network $N = (V,A)$ and nonnegative linear projection $\pi:\R^{A}\rightarrow\R^{d}$. 

For each vertex $u$ of $K_{n,n}$, the equation
$$
x(\delta(u)) = 1 \iff \sum_{e \in \delta(u)} x_e = 1
$$
is valid for $\PM(K_{n,n})$. From Lemma~\ref{lem:equations}, we conclude that for every node $v$ of $N$ there is a nonnegative vector $\epsilon^v\in\R^{2n}$ such that for every $s$--$v$ path $\sigma$ in the network $N$ and every vertex $u$ of the graph $K_{n,n}$ the following holds:
  \begin{equation*}
      \sum_{e\in\delta(u)}\pi_e(\chi^\sigma)=\epsilon_u^v\,.
  \end{equation*}
We base our analysis on the support of $\epsilon^v$, which we denote $\supp{\epsilon^v}$. 

Now consider a node $v$ of network $N$. For every $s$--$t$ path $\sigma$ going through $v$ and such that $\pi(\chi^{\sigma}) = \chi^{M}$ for some perfect matching $M$ of $K_{n,n}$, matching $M$ and cut $\delta(\supp{\epsilon^v})$ do not have an edge in common. 

Hence if $\sabs{\supp{\epsilon^v}}=n$ the $s$--$t$ paths of $N$ going through $v$ define at most $\frac{n}{2}!\frac{n}{2}!$ perfect matchings $M$ of $K_{n,n}$. 

Moreover, for every arc $a=(v_1,v_2)$ in $N$ with $\sabs{\supp{\epsilon^{v_1}}}=n_1<n$ and $\sabs{\supp{\epsilon^{v_2}}}=n_2>n$ there are at most $\frac{n_1}{2}!\frac{2n-n_2}{2}! \leqslant \frac{n}{2}! \frac{n}{2}!$ perfect matchings $M$  such that there is an $s$--$t$ path $\sigma$ in $N$ with $a \in \sigma $ and $\chi^{M}=\pi(\chi^{\sigma})$, since in this case $\sigma$ contains both nodes $v_1$ and $v_2$ and every such matching $M$ must contain all the edges from the support of $\pi(\chi^{\{a\}})$. 

Since the polytope $Q$ is an extension of $\PM(K_{n,n})$, for every perfect matching $M$ in $K_{n,n}$ there is an $s$--$t$ path $\sigma$ such that $\chi^\sigma$ projects to $\chi^M$. But since $\epsilon^{s}$ is an all zero vector and $\epsilon^{t}$ is an all one vector, this path $\sigma$ must go through a node $v$ with $\sabs{\supp{\epsilon^v}}=n$ or contain an arc $a = (v_1,v_2)$ with $\sabs{\supp{\epsilon^{v_1}}}<n<\sabs{\supp{\epsilon^{v_2}}}$.

Since the total number of perfect matchings in $K_{n,n}$ equals $n!$, network $N$ contains at least
\begin{equation*}
 \frac{n!}{2\frac{n}{2}!\frac{n}{2}!}=\Omega\left(\frac{2^n}{\sqrt{n}}\right)
\end{equation*}
 nodes $v$ with $\sabs{\supp{\epsilon^v}}=n$ or arcs $a=(v_1,v_2)$ with $\sabs{\supp{\epsilon^{v_1}}}<n<\sabs{\supp{\epsilon^{v_2}}}$. The result follows.
\end{proof}

The lower bound in Theorem~\ref{thm:bipartite_matchings} is tight, up to polynomial factors. Indeed, consider a complete bipartite graph $K_{n,n}$ with bipartition $U=\{u_1,\ldots, u_n\}$ and $W=\{w_1,\ldots, w_n\}$. We construct the network $N=(V,A)$ with
\begin{equation*}
  V:=2^W\qquad\text{and}\qquad A:=\setDef{(S_1,S_2)\in V\times V}{S_1\subseteq S_2\text{  and  }\sabs{S_1}+1=\sabs{S_2}}
\end{equation*}
and a linear projection $\pi$ such that for every arc $a=(S_1,S_2)\in A$
\begin{equation*}
    \pi_{u_i,w_j}(\chi^{\{a\}}):=\begin{cases}
                                        1 & \text{if}\quad i=\sabs{S_2},\quad \{w_j\}\cup S_1=S_2\\
					0 & \text{otherwise}
                                       \end{cases}\,.
\end{equation*}
It is not hard to see that every $\varnothing$--$W$ path in this network defines a perfect matching. This fact can be seen algorithmically, as follows. Start with $S=\varnothing$ and repeat the following step until $S = W$: having matched the vertices $v_1,\ldots,v_\sabs{S}$ with the vertices in $S$, select a mate $w \in W\setminus S$  for vertex $v_{\sabs{S}+1}$ and replace $S$ by $S\cup\{w\}$. It follows that the projection of the $\varnothing$--$W$ path polytope of network $N$ coincides with the perfect matching polytope of $K_{n,n}$. Since network $N$ has $n 2^{n-1} = O^*(2^n)$ arcs, we conclude that $\xc(\PM(K_{n,n})) = \Theta^*(2^n)$.

\subsection{Nonbipartite Perfect Matchings}

\begin{thm}\label{thm:complete_matchings}
Every uncapacitated flow-based extension (or, equivalently, $s$--$t$ path extension) of the perfect matching polytope of the complete graph $K_{n,n}$ has size $\Omega\left(\frac{2^{\frac{n}{2}}}{\sqrt{n}}\right)$.
\end{thm}
\begin{proof}
  Indeed, the polytope $\PM(K_{\frac{n}{2}, \frac{n}{2}})$ is a face of the polytope $\PM(K_n)$, and thus Lemma~\ref{lem:faces} gives the lower bound.
\end{proof}

In order to construct an $s$--$t$ path extension of size close to the lower bound in Theorem~\ref{thm:complete_matchings}, we consider a complete graph $K_{n}$ with vertex set $U=\{u_1,\ldots, u_n\}$ and construct the network $N=(V,A)$ with
\begin{align*}
  V:=\setDef{S\subseteq U}{\sabs{S}=2k,\ 0 \leqslant k\leqslant \frac{n}{2} \ \text{ and }\ \forall 1 \leqslant i\leqslant k : u_i\in S}\\ A:=\setDef{(S_1,S_2)\in V\times V}{S_1\subseteq S_2\text{  and  }\sabs{S_1}+2=\sabs{S_2}}
\end{align*}
and a linear projection $\pi$ such that for every arc $a=(S_1,S_2)\in A$
\begin{equation*}
    \pi_{u_i,u_j}(\chi^{\{a\}})=\begin{cases}
                                        1 & \text{if }\, \{u_i, u_j\}\cup S_1=S_2\\
					0 & \text{otherwise}.
                                       \end{cases}
\end{equation*}
It is once again easy to verify that this defines an $s$--$t$ path extension, this time of the perfect matching polytope of $K_n$. The idea is that every $\varnothing$--$U$ path in network $N$ defines a perfect matching of $K_n$ and conversely, every perfect matching of $K_n$ corresponds to at least one (actually many) $\varnothing$--$U$ path in $N$. The $\varnothing$--$U$ paths in $N$ actually correspond to perfect matchings whose edges are ordered in such a way that for each $i$, vertex $u_i$ is covered by one of the first $i$ edges in the ordering. Every arc $(S,S \cup \{u_i,u_j\})$ in such a path corresponds to the addition of edge $u_iu_j$ to the matching.


Up to a polynomial factor, the size of the network equals the number of nodes in the network, that is, 
\begin{equation*}
    \sum_{k=0}^{\frac{n}{2}} \binom{n-k}{k}\,.
\end{equation*}
This is due to the fact that the nodes $S$ in the $k$th level of network $N$ are of the form $S = \{u_1,\ldots,u_k\} \cup T$, where $T$ is contained in $U \setminus \{u_1,\ldots,u_k\}$ and has size $k$. Since the number of summands in the above expression is $\frac{n}{2}+1$, the size of the constructed extension is
\begin{equation*}
  O^*\left(\max_{0\leqslant k \leqslant \frac{n}{2}} \binom{n-k}{k}\right)=O^*\left(\max_{0 < k <\frac{n}{2}} \frac{(n-k)^{n-k}}{k^k (n-2k)^{n-2k}}\right)\,,
\end{equation*}
where we used Stirling's formula to simplify the left-hand side. Calculating the derivative of the function $\frac{(n-k)^{n-k}}{k^k (n-2k)^{n-2k}}$, we determine that the maximum in the above interval is achieved in the case when $k$ equals $\frac{2}{5+\sqrt{5}}n$, thus the size of the extension is $O(2^{0.695 n})$.

\subsection{Hamiltonian Cycles}

\begin{thm}\label{thm:complete_traveling}
  Every uncapacitated flow-based extension (or, equivalently, $s$--$t$ path extension) of the traveling salesman polytope of the complete graph $K_{n}$ has size $\Omega\left(\frac{2^{\frac{n}{4}}}{\sqrt{n}}\right)$.
\end{thm}
\begin{proof}
  Assume for now that $n=4k$ for some $k\in\N$, the other cases will be dealt with later. Take a partition of the vertices of $K_n$ in $U=\{u_1,\ldots, u_{2k}\}$ and $W=\{w_1,\ldots,w_{2k}\}$, and consider the following sets of edges in the graph $K_n$:
  \begin{equation*}
      E_0 := \setDef{u_iw_j}{i\neq j,\, 0\leqslant i,j\leqslant 2k}\qquad\text{and}\qquad E_1 := \setDef{u_iw_i}{ 0\leqslant i\leqslant 2k}\,.
  \end{equation*}
  Define the face $F$ of the polytope $\PT(K_n)$ as the set of points in $\PT(K_n)$ such that $x_e=0$ for every $e \in E_0$ and $x_e=1$ for every $e \in E_1$. 

  Let us show that the face $F$  together with an orthogonal projection on the variables corresponding to the edges $u_iu_j$ for $0\leqslant i,j\leqslant 2k$ gives an extension of the perfect matching polytope $\PM(K_{2k})$ (here the complete graph $K_{2k}$ is defined on the vertex set $U$).

  First, every Hamiltonian cycle $C$ in the graph $K_n$ restricted to the edges contained in $U$ is a perfect matching, whenever $\chi^{C}$ belongs to the face $F$. Indeed, for every vertex $u_i$ in $U$ there must be exactly two edges in $C$ adjacent to it. Since the characteristic vector $\chi^{C}$ lies in the face $F$, one of these edges is the edge $u_iw_i$ and the other is contained in $U$.

  Second, every perfect matching $M$ in the graph $K_{2k}$ can be extended to a Hamiltonian cycle $C$ in $K_n$ such that $\chi^{C}$ lies in $F$. Indeed, extend $M$ by another perfect matching $M'$ of $K_{2k}$ to a Hamiltonian cycle in $K_{2k}$. Then the desired hamiltonian cycle $C$ can be defined as the union of $M$, $E_1$ and $\setDef{w_iw_j}{u_iu_j \in M'}$. Thus the result follows from Theorem~\ref{thm:complete_matchings} and Lemma~\ref{lem:faces}.

  If $n=4k+r$, for some $k,r\in\N$, $1\leqslant r\leqslant 3$, the result is obtained in a similar way by taking a bipartition $U=\{u_1,\ldots, u_{2k}\}$ and $W=\{w_1,\ldots,w_{2k+r}\}$ and defining the face $F$ by equations $x_e=0$ for every $e \in E_0$, $x_e=1$ for every $e \in E_1$ and $x_{w_{2k}w_{2k+1}}=\ldots=x_{w_{2k+r-1}w_{2k+r}}=1$, where the edge sets $E_0$ and $E_1$ are defined as above.
\end{proof}

For the traveling salesman polytope there is a $s$--$t$ path extension of size $O^*(2^n)$ constructed in a similar manner as the $s$--$t$ path extension of the perfect matching polytope of $K_{n,n}$. This extension corresponds to a well-known dynamic programming algorithm of Held and Karp for the traveling salesman problem~\cite{heldKarp70}. We define this extension here for completeness.

Consider a complete graph $K_{n}$ with vertex set $U=\{u_1,\ldots, u_n\}$ and construct the network $N=(V,A)$ with
\begin{eqnarray*}
  V &:= &\setDef{(S,v)}{S\subseteq U,\, v\in S, u_1\in S}\cup\{(U,\varnothing)\}\\ 
  A &:= &\setDef{((S_1,v_1),(S_2,v_2))\in V\times V}{S_1\cup\{v_2\}=S_2\text{  and  }\sabs{S_1}+1=\sabs{S_2}}\\
  &&\cup\,\setDef{((U,v),(U,\varnothing))\in V\times V}{v\in U}
\end{eqnarray*}
and a linear projection $\pi$ such that for every arc $a=((S_1,v_1),(S_2,v_2))\in A$, $v_1\in U$, $v_2\in U$
\begin{equation*}
    \pi_{u_i,u_j}(\chi^{\{a\}}):=\begin{cases}
                                        1 & \text{if }\, \{u_i, u_j\}=\{v_1,v_2\}\\
					0 & \text{otherwise}
                                       \end{cases}
\end{equation*}
and for an arc $a=((U,v),(U,\varnothing))\in A$, $v\in U$
\begin{equation*}
    \pi_{u_i,u_j}(\chi^{\{a\}}):=\begin{cases}
                                        1 & \text{if }\,\{u_i, u_j\}=\{u_1,v\} \\
					0 & \text{otherwise}
                                       \end{cases}\,.
\end{equation*}
It is straightforward to see that the network with source $(u_1,\{u_1\})$ and sink $(U,\varnothing)$ generates the desired $s$--$t$ path extension.

\section{Open Problems}

We conclude this paper by stating three open problems.

\begin{enumerate}[(i)]
\item Obtain lower bounds for capacitated flow-based extensions. Although this type of extensions is more expressive than uncapacitated flow-based extensions, we suspect that exponential size lower bounds can be obtained for nonbipartite matchings and Hamiltonian cycles.
\item How difficult is this to compute a small uncapacitated flow-based extension for a given 0/1-polytope? Are there good general lower bounds?
\item All the lower bounds obtained here are of the type $2^{\Omega(\sqrt{d})}$, where $d$ is the dimension of $P$. Find an explicit $0/1$-polytope $P$ such that every uncapacitated flow-based extension has size $2^{\Omega(d)}$. (Notice that every polytope $P$ has an uncapacitated flow-based extension of size at most the number of vertices of $P$, thus this last lower bound would be essentially tight.)
\item Davis-Stober, Doignon, Fiorini, Glineur and Regenwetter~\cite{DDFGR13} give uncapacitated flow-based extensions of size $O^*(2^n)$ for the linear ordering polytope and $O^*(3^n)$ for the interval order polytope. Is there such an extension of size $O^*(c^n)$ for the semiorder polytope? (Semiorders are also known as unit interval orders.)
\end{enumerate}

\section{Acknowledgements}

The authors thank Hans Tiwary for taking part in the early stage of this work, and Michele Conforti, Santanu Dey, Marco Di Summa, Sebastian Pokutta and Dirk Oliver Theis for stimulating discussions.

\bibliographystyle{plain}
\bibliography{flow_extensions}

\end{document}